%% file: main.tex
\documentclass{article}
\pdfoutput=1
\usepackage{fullpage}
\usepackage[english]{babel}
\usepackage[cmex10]{amsmath}
 \usepackage{theoremref}
\usepackage[all=normal,paragraphs=tight,bibliography=tight]{savetrees}
\usepackage{todonotes}
\usepackage{amsthm}
\usepackage{amssymb,bbm}
\usepackage{microtype}
\usepackage{xspace}
\usepackage{todonotes}
\usepackage{comment}
\usepackage{enumerate}
\usepackage{tikz}
\usepackage{graphicx}
\usepackage[absolute]{textpos}
\usetikzlibrary{arrows, decorations.markings}
\usetikzlibrary{fit}					
\usetikzlibrary{backgrounds}

\usepackage{graphicx}
\usepackage{caption}
\usepackage{subcaption}

\numberwithin{equation}{section}
\newtheorem{theorem}{Theorem}[section]
\newtheorem{lemma}[theorem]{Lemma}

\newtheorem{conjecture}[theorem]{Conjecture}
\theoremstyle{definition}

\def\cqedsymbol{\ifmmode$\lrcorner$\else{\unskip\nobreak\hfil
\penalty50\hskip1em\null\nobreak\hfil$\lrcorner$
\parfillskip=0pt\finalhyphendemerits=0\endgraf}\fi} 

    \newcommand{\lab}[1]{\label{#1}}                
  \newcommand{\thlab}[1]{\thlabel{#1}} 

\renewcommand{\ge}{\geqslant} 
\renewcommand{\le}{\leqslant}
\newcommand{\se}{\subseteq}
\newcommand{\sm}{\setminus}

\newcommand{\size}[1]{\ensuremath{\left| #1 \right|}}

\newcommand{\eps}{\varepsilon}

\newcommand{\cA}{\ensuremath{\mathcal{A}}}
\newcommand{\cB}{\ensuremath{\mathcal{B}}}
\newcommand{\cF}{\ensuremath{\mathcal{F}}}
\newcommand{\cG}{\ensuremath{\mathcal{G}}}

\definecolor{myred}{RGB}{220,24,10}

\title{The Erd\H{o}s-Hajnal conjecture for caterpillars and their complements}
\author{%
Anita Liebenau\thanks{School of Mathematical Sciences, Monash University VIC 3800, Australia. 
Email: anita.liebenau@monash.edu. Research supported by a DECRA Fellowship from the 
Australian Research Council. The author would like to thank the Institute of Informatics, University of Warsaw, for its hospitality where this work was carried out.}
\and
Marcin Pilipczuk\thanks{%
Institute of Informatics, University of Warsaw, Poland. 
Email: marcin.pilipczuk@mimuw.edu.pl. 
The research of Marcin Pilipczuk is a part of projects that have received funding from the European Research Council (ERC) under the European Union's Horizon 2020 research and innovation programme
under grant agreement No 714704.}}

\date{}

\begin{document}

\maketitle

\begin{textblock}{20}(0, 12.5)
\includegraphics[width=40px]{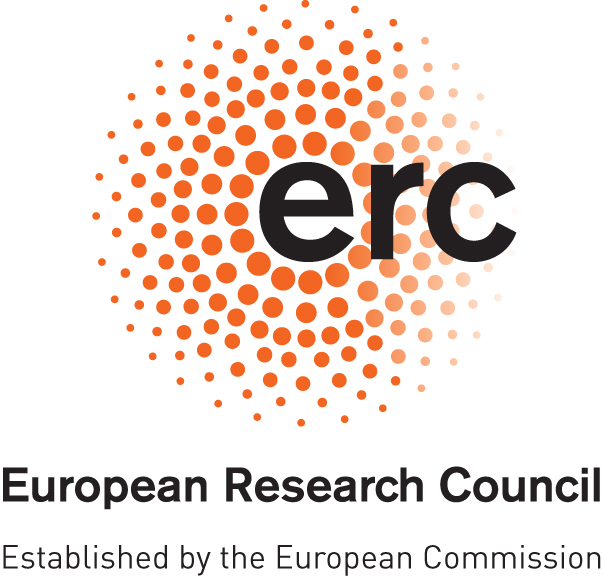}%
\end{textblock}
\begin{textblock}{20}(-0.25, 12.9)
\includegraphics[width=60px]{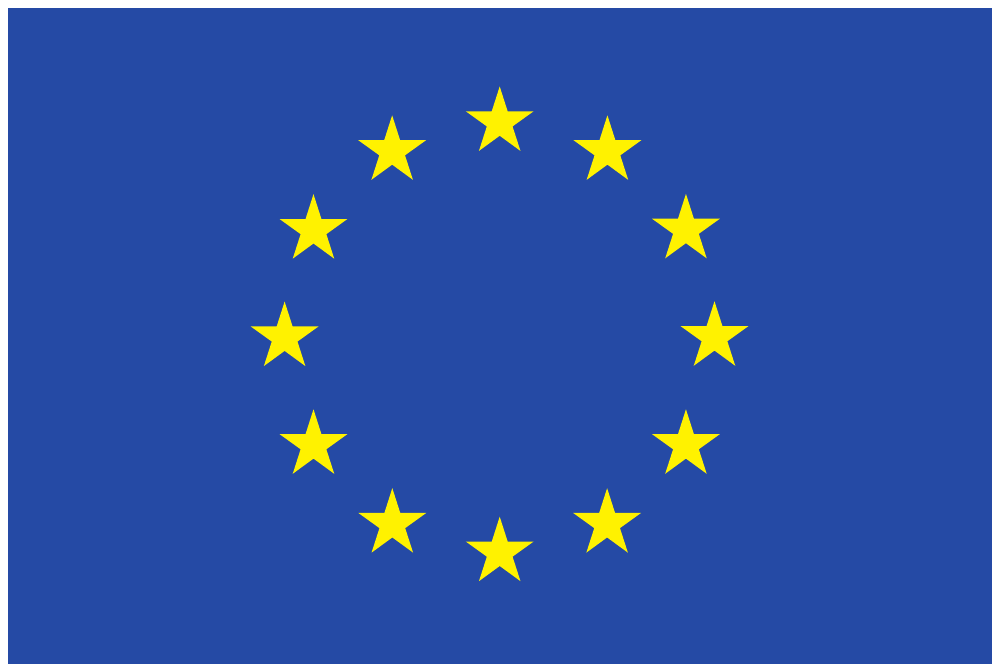}%
\end{textblock}

\begin{abstract}
The celebrated Erd\H{o}s-Hajnal conjecture states that for every proper hereditary graph class $\cG$
there exists a constant $\eps = \eps(\cG) > 0$ such that every graph $G \in \cG$ contains a clique or an independent set
of size $|V(G)|^\eps$. 
Recently, there has been a growing interest in the symmetrized variant of this conjecture, where one additionally 
requires $\cG$ to be closed under complementation.

We show that any hereditary graph class that is closed under complementation and excludes a fixed caterpillar as an induced subgraph
satisfies the Erd\H{o}s-Hajnal conjecture.
Here, a caterpillar is a tree whose vertices of degree at least three lie on a single path (i.e., our caterpillars may have arbitrarily long legs).
In fact, we prove a stronger property of such graph classes, called in the literature the \emph{strong} Erd\H{o}s-Hajnal property:
for every such graph class $\cG$, there exists a constant $\delta = \delta(\cG) > 0$ such that every graph $G \in \cG$ contains two disjoint sets $A,B \subseteq V(G)$
of size at least $\delta|V(G)|$ each so that either all edges between $A$ and $B$ are present in $G$, or none of them. 
This result significantly extends the family of graph classes for which we know that the strong Erd\H{o}s-Hajnal property holds; for graph classes excluding a graph $H$ and its complement
it was previously known only for paths~\cite{blt2015} and hooks (i.e., paths with an additional pendant vertex at third vertex of the path)~\cite{hooks}.
\end{abstract}

\section{Introduction}
\input{intro}

\section{Preliminaries}\label{sec:prelims}
\input{preliminaries}

\section{Proof of main theorem}\label{sec:proof}
\input{proof}

\section{Conclusions}\label{sec:conc}
\input{conclusions}

\paragraph*{Acknowledgements.}
The second author thanks Bartosz Walczak for long discussions on research directions outlined in the conclusions, in particular regarding $\chi$-boundedness of various graph classes.

\medskip

After this preprint was finished, we learned that Seymour and Spirkl~\cite{sophie} 
obtained independently the same result as Theorem~\ref{thm:main}.

\bibliographystyle{abbrv}
\bibliography{refs}

\end{document}

%% file: intro.tex
Classical bounds on the Ramsey numbers due to Erd\H{o}s and Szekeres~\cite{es1935} and Erd\H{o}s~\cite{e1947} imply that every
graph on $n$ vertices contains a clique or an independent set of size $\Omega(\log n)$, 
and this bound is tight up to the constant factor. 
A long-standing conjecture by Erd\H{o}s and Hajnal \cite{eh1977} states that this relation drastically changes
if attention is restricted to a fixed hereditary graph class $\cG$.

A graph class $\cG$ is called \emph{hereditary} if it is closed under taking induced subgraphs.
We say that a graph class $\cG$ has the \emph{Erd\H{o}s-Hajnal} property if there exists
a constant $\eps = \eps(\cG) > 0$ such that every $n$-vertex graph $G \in \cG$ contains a clique
or an independent set of size at least $n^\eps$. 
In 1989, Erd\H{o}s and Hajnal ~\cite{eh1989} proved a weaker bound of $2^{\Omega(\sqrt{\log n})}$ and conjectured
that every  proper hereditary graph class has the Erd\H{o}s-Hajnal property. 

Up to today, this conjecture remains widely open, and the Erd\H{o}s-Hajnal property remains
proved only for a limited number of graph classes $\cG$. We say a graph $G$ is {\em $H$-free} if $G$ does not 
contain the graph $H$ as an induced subgraph. It is known, for example, that 
the class of bull-free graphs has the Erd\H{o}s-Hajnal property~\cite{cs2008}. 
Yet, it is unknown whether the class of $P_5$-free graphs or the class of $C_5$-free graphs has the Erd\H{o}s-Hajnal property. 
The substitution procedure by Alon, Pach, and Solymosi~\cite{aps2001} provides more examples of graph classes with the 
Erd\H{o}s-Hajnal property. For a full overview, we refer to 
a survey of Chudnovsky~\cite{c2014}.

Substantial progress has been obtained by considering a weaker variant of the conjecture that asserts the Erd\H{o}s-Hajnal property for every proper hereditary graph class that is additionally closed under complementation. In this line of direction, a breakthrough paper of Bousquet, Lagoutte, and Thomass\'{e}~\cite{blt2015} proved the Erd\H{o}s-Hajnal property for graph classes excluding a fixed path and its complement. 
Bonamy, Bousquet, and Thomass\'{e} extended this result to graph classes closed under complementation that exclude long holes~\cite{bbt2016}, while Choromanski, Falik, Patel and the two authors extended the above result to {\em hooks}: paths with an additional pendant edge attached to the third vertex of the path~\cite{hooks}. 
Note that both the class of bull-free graphs and the class of $C_5$-free graphs are closed under complementation.

In fact, all the aforementioned results, 
except for the result for bull-free graphs,
prove a much stronger property of a graph class in question,
called the \emph{strong Erd\H{o}s-Hajnal property}. A graph class $\cG$ has the strong
Erd\H{o}s-Hajnal property if there exists a constant $\delta = \delta(\cG) > 0$ such that every
$n$-vertex graph $G \in \cG$ contains two disjoint sets $A,B \subseteq V(G)$ of size at least
$\delta n$ each, such that $A$ and $B$ are fully adjacent (there are all possible edges between
them) or anti-adjacent (there are no edges between them).
A simple inductive argument shows that a hereditary graph class with the strong Erd\H{o}s-Hajnal
property has also the (standard) Erd\H{o}s-Hajnal property (cf.~\cite{fp2008}).
Further graph classes known to obey the strong Erd\H{o}s-Hajnal property include
line graphs~\cite{hooks} and certain subclasses of perfect graphs~\cite{hooks,lt2016}.

Not all proper hereditary graph classes have the strong Erd\H{o}s-Hajnal property: 
a certain poset construction shows that this property is false for comparability graphs (and
    hence also for perfect graphs)~\cite{f2006}, and a simple randomized construction 
of~\cite{hooks} shows that if one considers the class of $\cF$-free graphs for a finite set
$\cF$ (i.e., graphs that do not contain any graph in $\cF$ as an induced subgraph),
then the strong Erd\H{o}s-Hajnal property is false unless $\cF$ contains a forest and a complement
of a forest.
This, together with the positive results of~\cite{blt2015,hooks}, motivates the following conjecture:
\begin{conjecture}\label{conj1}
For every tree $T$, the class of graphs excluding $T$ and its complement as induced subgraphs 
has the strong Erd\H{o}s-Hajnal property.
\end{conjecture}
In this work we prove this conjecture for $T$ being a {\em caterpillar} which is a tree in which all 
vertices of degree at least three lie on a single path (i.e., our caterpillars
    may have long legs).

\begin{theorem}\thlab{thm:main}
For every caterpillar $T$, the class of graphs excluding $T$ and the complement of $T$
as induced subgraphs has the strong Erd\H{o}s-Hajnal property.
\end{theorem}

In particular, Theorem~\ref{thm:main} resolves Conjecture~\ref{conj1} for every tree on at most $9$ vertices, 
a tree not covered by Theorem~\ref{thm:main} is depicted 
on Figure~\ref{fig:unknown}.
\begin{figure}[tb]
\begin{center}
\includegraphics{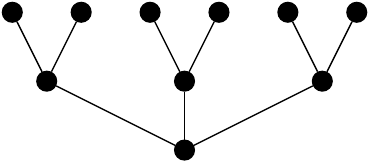}
\end{center}
\caption{The smallest, and the only 10-vertex tree,
  that is not a caterpillar, and thus is not covered by Theorem~\ref{thm:main}.}\label{fig:unknown}
\end{figure}

\medskip

The paper is organized as follows. After some preliminaries in Section~\ref{sec:prelims},
    we provide the main proof in Section~\ref{sec:proof}.
    In Section~\ref{sec:conc} we outline some directions for future research.

%% file: preliminaries.tex
We use standard graph theoretic notation. 
All graphs in this paper are finite and simple. 
Let $G$ be a graph, let $A,B\se V(G)$ such that $A\cap B = \emptyset$.  
By $G[A]$ we denote the subgraph in $G$ induced by the subset $A$. 
The set of vertices that have a neighbour in $A$ is denoted by 
$N(A)$, and $N[A] = N(A) \cup A$ denotes the closed neighbourhood of $A$. 
If $A = \{v\}$ we abbreviate and write $N(v)$ and $N[v]$, respectively. 
We also write $\deg(v,B)$ for $\size{N(v)\cap B}$. 
The set of edges with one endpoint in $A$ and one endpoint in $B$ is denoted by $E(A,B)$. 
If $E(A,B) = \emptyset$, we say that $A$ and $B$ are \emph{anti-adjacent}, and if $E(A,B) = A \times B$, then $A$ and $B$ are \emph{fully adjacent}.

For a constant $\eps > 0$, an \emph{$\eps$-pair} in a graph $G$ is a pair of two disjoint anti-adjacent sets $A,B \subseteq V(G)$ with $|A|,|B| \geq \eps |V(G)|$. 
Recall that a caterpillar is a tree in which all vertices of degree 
at least three lie on a single path. 
For integers $h,d,t \geq 1$ or $(h,d,t) = (1,0,0)$ we define a caterpillar $T_{h,d,t}$ 
as follows. 
Let $P^*$ be a path of length $h$, say on vertex set $\{v_1,\ldots,v_h\}$, 
and let $\{P_{i,j} : i\in [h], j\in [d]\}$ be a collection of pairwise vertex-disjoint paths 
and such that $V(P^*)\cap V(P_{i,j}) = \emptyset$ for all $i,j$. 
The graph $T_{h,d,t}$ is formed by taking the disjoint union of $P^*$ and the $hd$ 
paths $P_{i,j}$, and by adding the edges $v_iw_{ij}$, for all $i\in [h]$, $j\in[d]$, 
where $w_{ij}$ is one of the two endpoints of $P_{i,j}$. 
The graph $T_{1,0,0}$, for example, is the graph consisting of a single vertex, 
while $T_{1,1,t}$ is a path on $t+1$ vertices. 
Note that every caterpillar is an induced subgraph of $T_{h,d,t}$ for sufficiently large
parameters $h,d,t$.



We need the following two auxiliary results. The first one by Fox and Sudakov~\cite{fs2008} states that if a graph $G$ 
does not contain a graph $H$ as an induced subgraph then it contains either a very dense or a very sparse induced 
subgraph on a linear number of vertices. 
%
\begin{theorem}[\cite{fs2008}]\label{thm:fs}
For every integer $k$ and every constant $0 < \eps < 1/2$ there exists 
a constant $\delta$ such that every $n$-vertex graph $G$ satisfies the following.
\begin{itemize}
\item $G$ contains every graph on $k$ vertices as an induced subgraph;
\item there exists $A \subseteq V(G)$ of size at least $\delta |V(G)|$ such that
either $G[A]$ or the complement of $G[A]$ has maximum degree at most $\eps |A|$.
\end{itemize}
\end{theorem}

We need the following {\em path-growing argument} of Bousquet, Lagoutte, and Thomass\'{e}~\cite{blt2015}.
Since the paper~\cite{blt2015} states weaker bounds on the constants, we provide the proof for completeness.
\begin{lemma}[\cite{blt2015}]\thlab{lem:blt}
Let $t,\Delta \geq 1$ be integers.
Then, for every connected graph $G$ of maximum degree at most $\Delta$
and more than $(t+2)\Delta$ vertices, and for every $v \in V(G)$, either $G$ contains a pair $(A,B)$ of disjoint anti-adjacent vertex sets
of size at least $\Delta$ each, or an induced path
on $t$ vertices with one endpoint in $v$.
\end{lemma}
\begin{proof}
Assume that the graph $G$ does not contain the pair $(A,B)$ as in the lemma statement. 
This implies that for every set $C \subseteq V(G)$ of size at least $3\Delta$, the largest connected component of $G[C]$ has more
than $|C|-\Delta$ vertices.

We construct an induced path on vertices $v_1,v_2,\ldots$ inductively, starting with $v_1 = v$.
Given an integer $1 \leq i < t$ and vertices $v_1,v_2,\ldots,v_i$ that induce a path in $G$, we define $Z_i = N[\{v_1,v_2,\ldots,v_i\}]$ and $C_i$ to be the vertex set of the largest connected 
component of $G-Z_i$. Note that $|Z_i| \leq 1 + \Delta i$, and thus $|V(G) \setminus Z_i| \geq 3\Delta$. Consequently, $|C_i| \geq |V(G)|-1-\Delta i-\Delta \geq (t+1-i)\Delta$.
In the first step, we pick $v_2$ to be a neighbor of $v_1$ that is adjacent to a vertex of $C_1$; such a vertex exists by the connectivity of $G$.
In subsequent steps, we maintain the invariant that $v_i$ is adjacent to a vertex of $C_{i-1}$. This allows us always to choose the next vertex $v_{i+1}$ within the neighborhood of $C_i$,
completing the construction.

Since $|C_{t-1}| \geq 2\Delta$, the construction process does not stop before defining the vertex $v_t$, which gives the desired $t$-vertex path starting in $v$.
\end{proof}


%% file: proof.tex
\newcommand{\ferntree}{\ensuremath{T}}
\newcommand{\bud}{B}

Let $T$ be a caterpillar and let 
$h_0$, $d_0$, $t_0 \geq 1$ be such that $T$ is an induced subgraph of $T_{h_0,d_0,t_0}$. 
Let $\ell=\ell(h_0, d_0)$ be a large enough integer 
and let $\eps=\eps(h_0, d_0,t_0)>0$ be small enough. 
(We will determine $\ell$ and $\eps$ later.)
Let $G_0$ be a graph that does not contain $T$ nor the complement of $T$ 
as an induced subgraph. 
We show that $G_0$ contains an $\eps_0$-pair where $\eps_0=\eps_0(\eps,\ell,T)$ is small enough. 
By Theorem~\ref{thm:fs}, there exists a constant 
$\delta = \delta(T, \eps/\ell)$ and a set $A \subseteq V(G_0)$
of size at least $\delta |V(G_0)|$ such that either $G_0[A]$ or the complement
of $G_0[A]$ has maximum degree at most $\eps|A|/\ell$. 
Let $G$ be equal to $G_0[A]$ or the complement of $G_0[A]$, so that $G$ has maximum degree
at most $\eps|A|/\ell$. 

We partition arbitrarily the vertex set $V(G)$ into $\ell$ parts $V_1\dot\cup\ldots\dot\cup V_\ell$ 
such that $\left|\size{V_i}-\size{V_j}\right|\leq 1$ for all $i,j$.
Now, either $G$ contains an $\eps/2\ell$-pair, in which case we are done, 
  or the following holds:
\begin{itemize}

\item[$(a)$]
$\deg(v,V_i)\leq \eps\size{V_i}$ for all $v\in V(G)$ and all $1\le i\le \ell$, and 

\item[$(b)$] there are no two sets $A\se V_i$, $B\se V_j$, for some $1\le i,j \le n$, 
such that $\size{A}\geq \eps\size{V_i}$, $\size{B}\geq \eps\size{V_j}$ and 
such that $(A,B)$ forms an anti-adjacent pair. 

\end{itemize}

By discarding at most $\ell$ vertices of $G$ we may in fact assume that 
$V_1\dot\cup\ldots\dot\cup V_\ell$ is an equipartition, i.e.~that $\size{V_i}=\size{V_j}$ for all $i,j\in[\ell]$ 
and thus $\size{V_i} = \size{V(G)}/\ell$. 
Call a graph $G'$ {\em $\ell$-coloured} if the vertex set of $G'$ is partitioned into $\ell$ (not necessarily independent) 
sets $V_1\dot\cup\ldots\dot\cup V_\ell$; and call an $\ell$-coloured graph $G'$ {\em $\eps$-clean} if it satisfies 
(a) and (b).

We prove by induction on the maximum degree $d_0$ that an $\ell$-coloured $\eps$-clean graph $G'$ 
contains the {\em spine} 
of a caterpillar which can be extended to an induced copy of $T_{d_0,h_0,t_0}$, 
as long as $\ell$ is big enough and $\eps$ is small enough. 
To be precise, we need some notation.

For $\alpha > 0$ a set $\bud\se V(G)$ is called an \emph{$\alpha$-bud} 
if $G[\bud]$ is connected 
and $|N(\bud) \cap V_j| \geq \alpha |V_j|$ for some $j \in [\ell]$. 
A family $\cA$ of subsets of $V(G)$ is called {\em colour-compatible} 
if for all $A\in \cA$ there is an index $i_A\in [\ell]$ such that 
$A\se V_{i_A}$ and the indices $i_A$ for $A \in \cA$ are pairwise distinct.


For integers $h \geq 1$, $d \geq 0$,
and a constant $\alpha > 0$, an 
\emph{$(\alpha,h,d)$-junior caterpillar (in the underlying $\ell$-coloured graph $G$)} 
is a pair $(P,\cB)$, where $P$ is an induced path in $G$ on $h$ vertices 
$V(P) = \{v_1,v_2,\ldots,v_h\}$ and 
$\cB$ is a family of pairwise anti-adjacent $\alpha$-buds 
$\{\bud_{i,j} : i \in [h], j \in [d]\}$ such that 
the family 
$$\cB \cup \{\{v_i\} : i \in [h]\}$$ 
is colour-compatible and 
$N(\bud_{i,j}) \cap V(P) = \{v_i\}$ for all $i \in [h]$, $j \in [d]$. 
Note that colour-compatibility implies in particular that 
the $\bud_{i,j}$'s are pairwise disjoint and that 
$\bud_{i,j} \cap V(P) = \emptyset$ for all $i,j$. 
See Figure~\ref{fig:junior} for an illustration. 
\begin{figure}[tb]
\begin{center}
\includegraphics{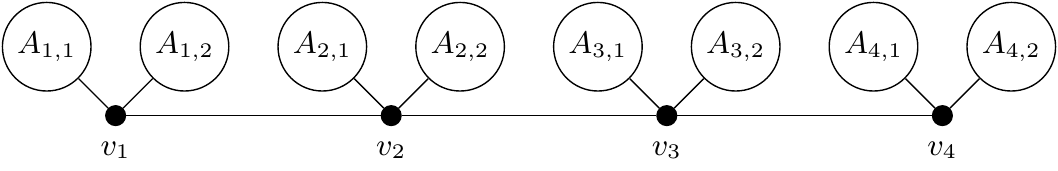}
\end{center}
\caption{A $(\alpha, h, d)$-junior caterpillar for $h=4$ and $d=2$.
Buds are depicted by empty circles.
  The buds are disjoint and pairwise anti-adjacent.
  The edges between buds and the path represent at least one edge;
a bud is not allowed to neighbor any other vertex $v_i$ on the path.}\label{fig:junior}
\end{figure}
We prove the following. 
\begin{lemma}\thlab{lem:main}
For all integers $h \geq 1$ and $d \geq 0$
there exist a constant $\alpha > 0$ and an integer $\ell_0 \geq 1$
such that every $\ell_0$-coloured $\alpha$-clean graph $G'$ contains 
an $(\alpha, h, d)$-junior caterpillar. 
\end{lemma}

Before proving the lemma let us show how this finishes the proof of \thref{thm:main}. 
By choosing $\ell \ge \ell_0(h_0,d_0)$ and $\eps < \alpha(h_0,d_0)$, 
where $\ell_0$ and $\alpha$ are the constants given by \thref{lem:main}, 
we may assume that $G$ contains an $(\alpha, h_0, d_0)$-junior caterpillar $(P,\cB)$. 
(We may fix $\ell = \ell_0(h_0,d_0)$ at this point, but we do need further restrictions on $\eps$.) 

Note that by definition, $P$ is an induced path and $\cB$ is a family of $h_0d_0$ pairwise disjoint 
$\alpha$-buds. 
%
%
%
Let $\bud_0^k$ for $k \in [h_0d_0]$ be an arbitrary enumeration of the buds $\bud\in \cB$ 
and let $v^k$ be the unique vertex of  $V(P)$ in $N(\bud_0^k)$. 
For every $k \in [h_0d_0]$, let $\bud^k$ be a minimal subset of $\bud_0^k$ with the following properties: 
$G[\bud^k]$ is connected, $v^k \in N(\bud^k)$, and 
$$|N[\bud^k]| \geq 10 h_0d_0t_02^k  \eps \size{V(G)}.$$ 
We claim that $\bud^k$ exists for all $k$. 
Indeed, since $\bud_0^k$ is an $\alpha$-bud we have that 
$\size{N[\bud_0^k]} \ge \alpha \size{V_j} \ge  \alpha \size{V(G)}/\ell$ for some $j\in [\ell]$. 
Hence, we can pick $ \bud^k$ vertex by vertex, starting with a neighbour of $v^k$, keeping 
$\bud^k$ connected, until $\size{N[\bud^k]}\ge 10 h_0d_0t_02^k  \eps \size{V(G)}$, 
which is possible if $\eps$ is chosen such that 
\begin{equation}\label{eq:epsCond}
10 h_0d_0t_02^k \ell \eps <  \alpha.
\end{equation}
We now show that $\size{N(\bud^k)}$ is not too big either. 
First note that 
by the minimality of $\bud^k$ and the bound on the maximum degree, we have that 
$$\size{N[\bud^k]} \leq \left(10 h_0d_0t_02^k +1\right) \eps \size{V(G)} < |V(G)|/2.$$ 
Now $(\bud^k, V(G)\setminus N[\bud^k])$ is an anti-adjacent pair. Hence, we can assume that 
$\size{\bud^k} < \eps \size{V(G)}$ for all $k \in [h_0d_0]$. 
It follows that for every $k \in [h_0d_0]$ we have that 
\begin{equation}\lab{eq:budk} 
\left(10 h_0d_0t_02^k -1 \right) \eps \size{V(G)} \leq \size{N(\bud^k)} \leq \size{N[\bud^k]} \leq  \left(10 h_0d_0t_02^k +1\right) \eps \size{V(G)}. 
\end{equation}
Therefore, for all $k \in [h_0d_0]$, 
\begin{align*}
\size{\bigcup_{j < k} N[\bud^j] \cup N[V(P)]} 
	&\le  \sum_{j < k} \size{N[\bud^j]}  + \size{N[V(P)]} \\
	&\le \left( (2^k-2) 10 h_0d_0t_0 + k + 2h_0\right) \eps \size{V(G)}\\
	&\le (2^k-1) 10 h_0d_0t_0 \eps \size{V(G)}, 
\end{align*}
where we used that $\size{N[V(P)]} \leq 2h_0\cdot \eps|V(G)|$ by the bound on the maximum degree, 
and that $k+2h_0 \le 3h_0d_0t_0.$ 
Hence, the sets $C^k$ defined by 
$$C^k = N(\bud^k) \setminus \left(\bigcup_{j < k} N[\bud^j] \cup N[V(P)]\right)$$ 
satisfy 
\begin{equation}\label{eq:Ck}
\size{C^k} \geq 9 h_0d_0t_0  \eps \size{V(G)}, 
\end{equation}
using \eqref{eq:budk} as well. 
By definition, 
$C^k \se N(\bud^k)$ 
and $C^k$ is anti-adjacent to $\bud^j$ for all $1 \leq j < k \leq h_0d_0$. 

Now, for every $k \in [h_0d_0]$ in \emph{decreasing order}, we define an induced path $P^k$ on $t_0+1$ vertices with 
$V(P^k)\se  \{v^k\} \cup \bud^k \cup C^k$ and endpoint $v^k$ as follows. 
First, we define
$$D^k = \{v^k\} \cup \left(\left(\bud^k \cup C^k\right) \setminus \bigcup_{j > k} N[V(P^j) \setminus \{v^j\}]\right).$$
Since $\bud^k$ and $A^j\cup C^j$ are anti-adjacent for $1 \leq k < j \leq h_0d_0$ and since 
$V(P^j)\se  \{v^j\} \cup \bud^j \cup C^j$, 
we have that $\{v^k\} \cup \bud^k \subseteq D^k$ and, consequently, $G[D^k]$ is connected. 
Note that 
$$\size{\bigcup_{j > k} N[V(P^j) \setminus \{v^j\}]}\leq  h_0d_0t_0 \eps  \size{V(G)}$$ 
so that 
$$\size{D^k} \geq 8 h_0d_0t_0 \eps\size{V(G)},$$
by~\eqref{eq:Ck}. 
Finally, since the maximum degree in $G[D^k]$ is at most $\eps \size{V(G)}\leq \size{D^k}/(t_0+2)$, 
we can apply \thref{lem:blt} to $G[D^k]$ to find an induced path $P^k$ on $t_0+1$ vertices rooted in $v^k$. 

From the construction, we infer that every path $P^k$ is induced with an endpoint
in $v^k$, $v^k$ is the unique
vertex of $V(P) \cap N[V(P^k) \setminus \{v^k\}]$, and $V(P^k)\setminus \{v^k\}$
and $V(P^j) \setminus \{v^j\}$ are anti-adjacent for every $1 \leq k < j \leq h_0d_0$.
Consequently, $V(P) \cup \bigcup_{k \in [h_0d_0]} V(P^k)$ induces the 
caterpillar $T_{h_0,d_0,t_0}$ in $G$. 

This finishes the proof of \thref{thm:main}. It remains to prove \thref{lem:main}.

\medskip 

We remark that the above proof does not use the assumptions on the parts of a 
junior caterpillar being colour-compatible; colour-compatibility will be used only 
in the inductive step in the proof of \thref{lem:main}.  
In that proof we will iteratively find a more general structure which we call ferns. 
If such a fern is ``deep'' enough then it is easy to see that it contains the desired junior caterpillar. 
Otherwise, we construct the junior caterpillar inductively. 

Let us recall some notation. 
A \emph{rooted tree} is a tree $T$ with one vertex specified as the \emph{root}. 
We refer to vertices of such trees as \emph{nodes} as the rooted trees in this proof are 
auxiliary structures outside of the graph we are working with. 
For a node $t \in V(T)$, the subtree of $T$ rooted at $t$ is denoted by $T_t$. 
If $s$ is the neighbour of $t$ in $T$ on the path from the root to $t$ then $s$ is called 
{\em the parent of $t$} whereas $t$ is called {\em a child of $s$.} 
The \emph{height} of a rooted tree $T$ is the maximum number of edges on a path from the root 
to a leaf.

For a constant $\alpha > 0$ and an integer $d \geq 1$, an \emph{$(\alpha, d)$-fern} (in an underlying coloured graph $G'$) is a pair $F=(\ferntree,\cB)$ where 
$\ferntree$ is a rooted tree in which every 
internal node has exactly $d$ children, 
and where $\cB=\{\bud_t : t\in V(\ferntree)\}$ is a colour-compatible family of 
$\alpha$-buds, one for each node in $\ferntree$, that are pairwise anti-adjacent 
unless $s$ is a parent of $t$ in $\ferntree$ in which case $A_s\se N(A_t)$, 
i.e.~every $v\in A_s$ has a neighbour in $A_t$. 
For a given fern $F=(\ferntree,\cB)$, we refer to $\ferntree$ as $\ferntree(F)$, 
to $\cB$ as $\cB(F)$, and we set $V(F) = \bigcup_{\bud\in\cB}\bud$ to be the vertex 
set on which the fern lives. 
The \emph{height} of a fern $(\ferntree,\cB)$ is the height of the tree $\ferntree$. See Figure~\ref{fig:fern} for an illustration.
\begin{figure}[tb]
\begin{center}
\includegraphics{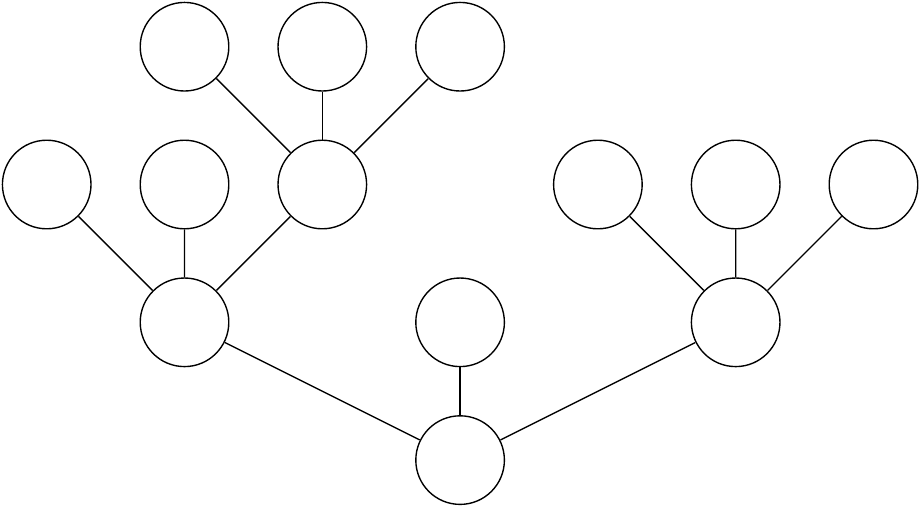}
\end{center}
\caption{A $(\alpha, d)$-fern for $d=3$.
Buds are depicted by empty circles.
  A line connecting two buds represents that every vertex from a lower bud has at least one neighbor in the upper bud.
  All other pairs of buds are anti-adjacent.}\label{fig:fern}
\end{figure}
Finally, we say that a fern $(\ferntree, \cB)$ \emph{grows} on a set $Z \subseteq V(G')$ 
if $Z$ is disjoint from every bud $\bud_t\in\cB$, 
$Z$ is anti-adjacent to every bud $\bud_t$ for non-root nodes $t$, 
and $Z \subseteq N(\bud_r)$ for the root $r$ of $\ferntree$. 

It is easy to observe that a high enough fern contains a junior caterpillar.
\begin{lemma}\thlab{lem:fern2cat}
Let $\ell,h \geq 1$ and $d \geq 0$ be integers, and let $\alpha > 0$ be a constant.
If an $\ell$-coloured graph $G'$ contains an $(\alpha,d+1)$-fern of height at least $h$, 
then $G'$ contains an $(\alpha,h,d)$-junior caterpillar. 
\end{lemma}
\begin{proof}
Let $(\ferntree, \cB)$ be an $(\alpha,d+1)$-fern of height at least $h$. 
By the assumption on height, there exists a path $P$ in $\ferntree$, say on nodes  
$\{t_1,t_2,\ldots,t_{h+1}\}$, such that $t_1$ is the root of $T$ and $t_{i+1}$ is a 
child of $t_i$ for all $i\in [h]$. 
Iteratively, for every $i \in [h]$ in increasing order, 
pick $v_i \in \bud_{t_i}$ such that $v_iv_{i-1} \in E(G)$ for $i > 1$;
this is always possible since $N(\bud_{t_i}) \supseteq \bud_{t_{i-1}}$. 
Observe that the adjacency conditions on a fern imply that $\{v_1,v_2,\ldots,v_h\}$
induce a path in $G'$ with vertices in this order.
To complete the construction of a junior caterpillar,
for every $i \in [h]$, define $\alpha$-buds $\bud_{i,j}$, $j \in [d]$, to be the $d$ buds 
$\bud_s$ for children $s$ of $t_i$ except for $t_{i+1}$. 
Finally, observe that the adjacency conditions on a fern imply that $\bud_{i,j}$ are pairwise
anti-adjacent with $N(\bud_{i,j}) \cap \{v_1,v_2,\ldots,v_h\} = \{v_i\}$, while 
colour-compatibility of $\cB$ implies colour-compatibility of 
$\{B_{i,j} : i \in [h], j \in [d]\} \cup V(P)$. 
\end{proof}

%
%
We are now ready to prove our main technical lemma.
\begin{proof}[Proof of \thref{lem:main}]
We prove the statement by induction on $d$. Fix integers $h \geq 1$ and $d \geq 0$
and, if $d > 0$, assume that the statement is true for $h$ and $d-1$ with parameters
$\alpha'=\alpha(h, d-1) > 0$ and $\ell'=\ell_{0}(h, d-1)$.
%
Let $G'$ be an $\ell_0$-coloured $\alpha$-clean graph 
where 
\begin{align}
\ell_0  &= h+ \ell'  \cdot \left( (d+1)^{h+1} +1\right)\label{eq:def-ell}\\
\alpha &= 3^{-\ell_{0}} \cdot \alpha' / 10,\label{eq:def-eps}
\end{align}
and where, for brevity, we set $\alpha'=\ell'=1$ in the case $d=0$. 
Assume that $V_1, \ldots, V_{\ell_0}$ are the colour classes of $G'$. 

We now describe an iterative process in which we find a family $\cF$ of ferns in $G'$,  
increasing the total number of buds by 1 in each step. 
For $k=0,1,\ldots$, we define a set $I^k \se I^{k-1}\se\ldots\se [\ell]$ of $\ell - k$ active colours, 
subsets $W^k_i \se V_i$ for all $i\in I^k$, 
and a family $\cF^k$ of $(\alpha, d+1)$-ferns, 
each one growing on $W^k_i$ for some $i\in I^k$. 
See Figure~\ref{fig:induction} for an illustration. 
\begin{figure}[tb]
\begin{center}
\includegraphics{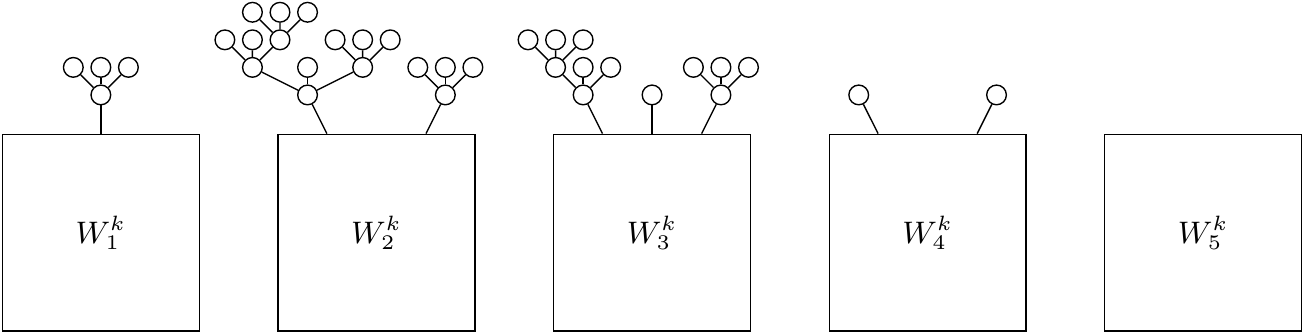}
\end{center}
\caption{A state after $k$-th step of the inductive process for $d=2$.
As before, buds are represented by circles and an edge between two sets 
means that every vertex in the lower set has at least one neighbor in the upper set.
Rectangles depict sets $W_i^k$ which are of size linear in the size of their original color classes.
There may be edges between different sets $W_i^k$; except for these and the drawn edges,
every other pair of depicted sets is anti-adjacent.
Next step will be performed inside the set $a=3$, as this set has $d+1$ ferns growing on it.}\label{fig:induction}
\end{figure}
In each step, let $\cF_i^k$ be the set of all those ferns $F\in \cF^k$ that grow on $W^k_i$, 
and let $\cB^k=\bigcup_{F\in \cF^k} \cB(F)$ be the collection of all buds in all ferns of 
$\cF^k$. 


\bigskip
\noindent
{\bf The procedure.\\}
We set $I^0 = [\ell]$, $W^0_i = V_i$, and $\cF_i^0 = \emptyset$ for every $i \in I^0$. 
Now assume that for $k\geq 0$ we have sets $I^k$, $(W^k_i)_{i\in I^k}$, 
and a family $\cF^k$ of ferns as above. 

\medskip

\noindent
We stop the process if for all 
$i\in I^k$ we have that $0< \size{\cF_i^k} < d+1$, 
or if $\size{I^k}=1$. 

\medskip

\noindent
Otherwise, we define $I^{k+1}$, $(W^{k+1}_i)_{i\in I^{k+1}}$, 
and $\cF^{k+1}$ as follows. 
Pick an index $a \in I^k$ as follows. 
If there exists a set $\cF_i^k$ of size $d+1$, pick $a = i$; 
otherwise, pick $a$ so that $\cF^k_a = \emptyset$. 
Now, let $A_0$ be the vertex set of the largest connected component of $G'[W^k_a]$.
Let $A \subseteq A_0$ be a minimal subset 
such that $G'[A]$ is connected and there exists an index $b \in I^k \setminus \{a\}$ 
with $|N(A) \cap W^k_b| \geq |W^k_b|/3$.
Set 
\begin{align*}
I^{k+1} & = I^k \setminus \{a\},\\
W^{k+1}_b &= W^k_b \cap N(A),\\
W^{k+1}_i &= W^k_i \setminus N(A) \text{ for every } i \in I^{k+1} \setminus \{b\}\text{, and} \\
\cF_i^{k+1} &= \cF_i^k \text{ for every } i \in I^{k+1} \setminus \{b\}.
\end{align*}
Finally, set 
$\cF_b^{k+1} = \cF_b^k \cup \{F_A\}$ for a fern $F_A=(\ferntree, \cB_A)$ defined as follows. For $f = |\cF_a^k| \in \{0,d+1\}$, suppose that  
$F_1,F_2,\ldots,F_f$ are the ferns in $\cF_a^k$ (growing on $W_a^k$), 
and suppose that $F_i = (\ferntree_i, \cB_i)$ for $i \in [f]$, 
and let $r_i$ be the root of $\ferntree_i$. 
We create a new node $r$ to be the root of $\ferntree$, associate with it the set $A$, 
and append all trees $T_i$ as subtrees rooted at the children of $r$; 
that is, every node $r_i$ becomes a child of $r$. 
We then define $\cB_A = \{ A\} \cup \bigcup_{i \in [f]} \cB_i$. 

\bigskip

Note that for $k=0$ and for the initial choices $I^0$, $W^0_i$, and $\cF^k$, 
the following hold trivially.  
\begin{enumerate}
\item $|W^k_i| \geq 3^{-k} |V_i|$.

\item Every $F\in \cF^k$ is indeed an $(\alpha, d+1)$-fern, 
growing on $W^k_i$ for some $i\in I^k$. 

\item $\size{\cB^k} = k$ and $\cB^k\cup \{W_i^k : i \in I^k\}$ is colour-compatible. 
\item $W^k_i$ is anti-adjacent to all buds $\bud \in \cB(F)$ of ferns $F\in \cF^k\sm \cF_i^k$.
\item $V(F)$ and $V(F')$ are anti-adjacent for all distinct ferns $F,F'\in \cF$. 
\item $\size{\cF_i^k}\leq d+1$, and $\size{\cF_i^k} = d+1$ for at most one family $\cF_i^k$.
\end{enumerate}

Assume now that for $k\geq 0$ we have sets $I^k$, $(W^k_i)_{i\in I^k}$, 
and a family $\cF^k$ of ferns that satisfy the conditions 1.-6. 
We claim that, unless the process stops, 
the procedure is well-defined and that the properties 1.-6.~hold for 
$I^{k+1}$, $(W^{k+1}_i)_{i\in I^{k+1}}$, and $\cF^{k+1}$.  

Note that $k<\ell-1$ since otherwise the process would stop (since $\size{I^k}=\ell-k$). 
Furthermore, we can assume that $\size{\cF_i^k} \in\{0,d+1\}$ for some 
$i\in I^k$ by the stopping condition, and hence the index $a$ is well-defined. 
Note also that there is at most one family with $\size{\cF_i^k} = d+1$ by 6., and that the 
index $a$ picks this unique family if it exists. 
By 1., 
$\size{W^k_a} \geq 3^{-k} \size{V_i}$, and hence, 
$\size{A_0} > \size{W^k_a}/2$, as otherwise 
$G'$ contains an anti-adjacent pair $(U,W)$, 
both $U,W\se V_a$ and of size at least 
$\size{V_a}/2\cdot 3^k > \alpha \size{V_a}$, 
contradicting $G'$ being $\alpha$-clean. 
This then implies that 
$\size{N(A_0) \cap W^k_i} \geq \size{W^k_i}/2$ for all $i\in I^k \setminus \{a\}$ 
as otherwise $(A_0, W^k_i \setminus N(A_0))$ 
would be an anti-adjacent pair contradicting $G'$ being 
$\alpha$-clean. 
Consequently, since $I^k \setminus \{a\} \neq \emptyset$, 
the set $A$ as defined exists. 
It follows that $I^{k+1}$, $(W^{k+1}_i)_{i\in I^{k+1}}$ and $(\cF_i^{k+1})_{i\in I^{k+1}}$ 
are well-defined. 

By definition, $\size{N(A) \cap W^k_b} \geq |W^k_b|/3\ge 3^{-{(k+1)}}\size{V_b}$. 
Furthermore, we have 
that 
$\size{N(A) \cap W^k_i} \le  \size{W^k_i}/3 + \alpha \size{V_i} \leq 2\size{W^k_i}/3$
for all $i\in I^{k+1}\sm \{b\}$, 
by minimality of $A$ and the upper bound on $\deg(v,V_i)$ for every 
$v\in A$ (by part (a) of the definition of $G'$ being $\alpha$-clean). 
Hence, for all $i \in I^{k+1}$, $\size{W^{k+1}_i} \geq \size{W^k_i}/3 \geq 3^{-{(k+1)}}\size{V_i}$, and thus 1.~holds. 

We now show that $F_A=(\ferntree, \cB_A)$ is an $(\alpha,d+1)$-fern. By definition, $T$ is a rooted tree. If $\cF_a^k=\emptyset$ then 
$T$ consists of a single node $r$ (which is the root). Otherwise, $\size{\cF_a^k}=d+1$ 
and the trees $T_1,\ldots,T_{d+1}$ are appended as subtrees rooted at the children of $r$. Thus, the degree of $r$ is $d+1$ in $T$. 
The family $\cB_A = \{ A\} \cup \bigcup_{i \in [f]} \cB_i$ is colour-compatible 
since $A\se W_a^k$ and $\cB^k\cup \{W_i^k: i\in I^k\}$ is colour-compatible by 3. 
Furthermore, $A$ is an $\alpha$-bud since 
$\size{N(A) \cap W^k_b} \geq \alpha \size{V_b}$. 
Note that $A$ is anti-adjacent to $\bud$ for any bud 
$\bud\in \bigcup_{i \in [f]} \cB_i$ unless $\bud$ is a bud corresponding to 
the root of $T_i$ for some $i \in [f]$. 
This is because $A\se W_a^k$ and each $F_i$ grows on $W_a^k$. 
For the same reason, $A\se N(\bud)$ for every $\bud$ corresponding to 
the root of $T_i$ for some $i \in [f]$. 

We now claim that $F_A$ grows on $W_b^{k+1}$. 
Indeed, $W_b^{k+1}\se W_b^k$ is disjoint from $V(F_A)$ 
since $\cB^k\cup \{W_i^k: i\in I^k\}$ is colour-compatible by 3. 
By definition, $W_b^{k+1}\se N(A)$. 
Finally, $W_b^{k+1}\se W_b^k$ is anti-adjacent to $\bud$ for any 
$\bud\in \bigcup_{i \in [f]} \cB_i$ since $F_i\in \cF_a^k$ and by 4. 
This proves 2.~for $k+1$.  
For 3.~note that $\cB^{k+1} = \cB^{k}\cup \{A\}$ and that 
colour-compatibility of $\cB^{k+1} \cup \{W_i^{k+1} : i \in I^{k+1}\}$ 
follows from $W_i^{k+1}\se W_i^k$ for all $i\in I^{k+1}$, 
from $A\se W_a^k$, and from colour-compatibility of 
$\cB^{k} \cup \{W_i^{k} : i \in I^{k}\}$.  
For 4., note that $(W_i^{k+1}, V(\bud))$ is anti-adjacent 
for all $i\in I^{k+1}$, $\bud\in \cB\left(\cF^{k+1}-(\cF_i^{k+1}\cup \{A\})\right)$, 
by 4.~for $k$. 
By definition, $A$ is anti-adjacent to all $W_i^{k+1}$ for $i\neq b$ 
and 4.~follows. 
Similarly, we only need to check 5.~for $A$. But this follows 
since $A\se W_a^k$ and from 4.~for $k$. 
Finally, in step $k+1$ we add 1 fern to the family $\cF_b^k$ 
growing on $W_b^{k+1}$ (and to no other family), and the 
choice of $a$ guarantees that there is at most 
one family $\cF_i^{k+1}$, $i \in I^{k+1}$, with exactly $d+1$ ferns. 

\medskip

%

Let us consider the case when $d=0$ first. Note that 
for every $k\ge 0$, every family $\cF_i^k$ has size 0 or 1, for 
$i\in I^k$. So the process only stops after step $k=\ell_0-1$ when 
$\size{I^k}=1$ and the process produces a fern in which 
the auxiliary tree is a path of length $\ell_0 -1 \geq h$. 
That is, the fern is an $(\alpha,1)$-fern of height at least $h$, 
which contains an $(\alpha,h,0)$-junior caterpillar by \thref{lem:fern2cat}. 
This proves the statement for $d=0$. 

For $d\ge 1$, 
assume now that the process stops at step $k$ for some $0\leq k\leq \ell_0$, 
that is, $\size{I^k} = 1$ or $1\leq \size{\cF_i^k}\leq d$ for all $i\in I^k$. 
If there exists a fern $F \in \cF^k$ of height at least $h$ then 
$G'$ contains an $(\alpha, h, d)$-junior caterpillar by \thref{lem:fern2cat}, 
and we are done.
So we may assume that every fern $F \in \cF^k$ is of height less than $h$, 
and thus it contains at most $(d+1)^h$ buds. 
Since every set $\cF_i^k$ is of size at most $d$, 
and the total number of buds in $\cB$ is exactly $k$, 
we have that
$k =\size{\cB}\leq \size{I^k}  (d+1)^{h+1}.$  
Now since $k= \ell_0 - \size{I^k}$ as well, we infer that 
$\size{I^k}\geq \ell_0/\left((d+1)^{h+1}+1\right) \ge \ell',$
by~\eqref{eq:def-ell}. 

Pick any set $I \subseteq I^k$ of size exactly $\ell'$ and consider an $\ell'$-coloured 
graph $H = G'[\bigcup_{i \in I} W^k_i]$ with colour classes $W^k_i$, $i \in I$.
Since $G'$ is $\alpha$-clean and $\size{W^k_i} \geq 3^{-k} \size{V_i} \geq 3^{-\ell_0} \size{V_i}$, we have
that $H$ is $\alpha'$-clean by~\eqref{eq:def-eps}.
By induction, there exists an $(\alpha', h, d-1)$-junior caterpillar $(P, \cB')$ in $H$. 
Let $V(P) = \{v_1,v_2,\ldots,v_h\}$ 
and let $\cB' = \{B_{i,j} : i\in [h], j\in [d-1]\}$ be the family of $\alpha'$-buds 
so that $\cB'\cup\{\{v_i\} : i\in [h]\}$ is colour-compatible in $H$ and so that 
$\{ B_{i,j} : j\in [d-1]\}$ is the set of {\em private} $\alpha'$-buds for vertex $v_i$ (in $H$). 
Note that an $\alpha'$-bud in $H$ is an $\alpha$-bud in $G'$, since 
$\alpha'\size{W^k_i} \geq \alpha \size{V_i}$. 
Note also that colour-compatibility in $H$ implies colour-compatibility in $G'$. 

For $i \in [h]$, let $a(i)$ be such that $v_i \in W^k_{a(i)}$, 
let $F_{i}$ be any fern in $\cF^k_{a(i)}$ 
(which exists as $\size{\cF^k_{a(i)}}\ge 1$), 
and let $B_{i,d}\in \cB(F_i)$ be the bud associated with the root node of $F_i$. 
Note that $B_{i,d}$ is anti-adjacent to all $W_{j}^k$ for $i\neq a(i)$, 
by Property 4, and hence it is anti adjacent to all $B_{i',j}$ for $i'\neq i$, all $j\in [d-1]$. 
Furthermore, $N(B_{i,d})\cap V(P) = \{v_i\}$, and $B_{i,d}$ is anti-adjacent 
to $B_{i',d}$ for any $i\neq i'$ by Property 5. 
Hence, $(P,\cB^*)$ is an $(\alpha, h, d)$-junior caterpillar in $G'$, 
where $\cB^*=\{B_{i,j} : i \in [h], j \in [d]\}$.
This finishes the proof of the lemma.
\end{proof}

%% file: conclusions.tex
We would like to conclude with a number of future research directions.

Our result inscribes into a recent trend in studying the (strong) Erd\H{o}s-Hajnal property for graph classes closed under complementation.
Here, the progress is triggered by the result of Fox and Sudakov~\cite{fs2008}, see Theorem~\ref{thm:fs}, which allows to concentrate on sparse graphs. 
For a fixed graph $H$, let $\cG_H$ be the class of graphs that does not contain $H$ nor the complement of $H$ as an induced subgraph.
As shown in~\cite{hooks}, a simple randomized construction shows that $\cG_H$ does not have the strong Erd\H{o}s-Hajnal property unless $H$ is a forest or the complement of a forest. 
In this work, we significantly extend the family of $H$ for which $\cG_H$ is known to have the strong Erd\H{o}s-Hajnal property from paths~\cite{blt2015} and hooks~\cite{hooks} to arbitrary caterpillars. 
We believe that Conjecture~\ref{conj1} is true in general. 
%
%
The natural next step would be to prove the conjecture for the tree depicted in Figure~\ref{fig:unknown}. 

To prove the (strong) Erd\H{o}s-Hajnal property for graph classes closed under complementation 
similar techniques are used as for proving $\chi$-boundedness. 
A graph class $\cG$ is called \emph{$\chi$-bounded} if for every $k$ there exists a constant $c$ such that every graph $G \in \cG$ with chromatic number at least $c$
contains a complete subgraph on $k$ vertices. 
The well-known Gy\'arf\'as-Sumner conjecture~\cite{gyarfas,sumner} asserts that for every tree $T$, the class of $T$-free graphs is $\chi$-bounded. The proof of Bousquet, Lagoutte, and Thomass\'{e}~\cite{blt2015} can be considered as the Erd\H{o}s-Hajnal analogue of the proof of Gy\'arf\'{a}s that graphs excluding a fixed path are $\chi$-bounded. 

In the world of $\chi$-boundedness, an old result of Scott~\cite{scott1997} proves that if one excludes a fixed tree $T$ and \emph{all its subdivisions}, then one obtains 
a $\chi$-bounded graph class. A weakening of Conjecture~\ref{conj1} would be the following. 
\begin{conjecture}\label{conj:subdiv}
For every fixed tree $T$, the set of graphs $G$ such that neither $G$ nor the complement of $G$ contains a subdivision of $T$ as an induced subgraph, has the strong Erd\H{o}s-Hajnal property.
\end{conjecture}
The proof in~\cite{scott1997} 
considers two cases, depending on whether constant-radius balls have small or large chromatic number.
In the world of the Erd\H{o}s-Hajnal property, this translates to whether the following property is true:
\begin{quote}
For some radius $r \geq 1$ and a constant $\eps > 0$, every set $A \subseteq V(G)$ of size at least $\eps |V(G)|$
contains a vertex $v \in A$ such that in $G[A]$, at least $\eps |A|$ vertices of $A$ are within distance at most $r$ from $v$.
\end{quote}
If a graph $G$ has the above property for some fixed $r$ and $\eps$,
then we say that $G$ has the \emph{large balls} property, which corresponds to the large local chromatic number case of~\cite{scott1997}.
In a hypothethical proof of Conjecture~\ref{conj:subdiv}, we believe that one can handle the case of the large balls property with the techniques of this paper.
However, if this property is false, we do not know how to proceed. In Scott's work~\cite{scott1997}, the main tool in the small local chromatic number regime is a Ramsey-like lemma (Lemma 5 of~\cite{scott1997}), which has no good analogue in the Erd\H{o}s-Hajnal setting.

Finally, it would be tempting to use the approach of this paper to make a progress of the Gy\'arf\'as-Sumner conjecture, say to prove it for any caterpillar $T$. 
While many parts of the proof can be translated to the $\chi$-boundedness regime, one step specific to the Erd\H{o}s-Hajnal setting seems is an obstacle:
in the proof of Lemma~\ref{lem:main}, we argued that $N(A_0)$ contains most of every set $W^k_i$ for $i \neq a$, and therefore the set $A$ and the index $b$ exist;
in the $\chi$-boundedness setting, it is not clear how to proceed if $N(W^k_a) \cap W^k_i$ has small chromatic number for every $i,a \in I^k$, $i \neq a$.